\documentclass[11pt, reqno]{amsart}
\usepackage{macros}

\title{Fourier--Mukai transforms commuting with Frobenius}
\author{Daniel Bragg}
\address{Department of Mathematics, University of Utah, Salt Lake City, UT 84112}
\email{bragg@math.utah.edu}

\begin{document}


\begin{abstract}
    We show that a Fourier--Mukai equivalence between smooth projective varieties of characteristic $p$ which commutes with either pushforward or pullback along Frobenius is a composition of shifts, isomorphisms, and tensor products with invertible sheaves whose $(p-1)$th tensor power is trivial.
\end{abstract}

\maketitle

\section{Introduction}

Let $\X$ be a smooth projective variety over an algebraically closed field $k$ of positive characteristic $p$. We write $\D^{\mathrm{b}}(\X)$ for the bounded derived category of coherent sheaves on $\X$, viewed as a $k$-linear triangulated category. The \emph{absolute Frobenius morphism} of $\X$ is the morphism $\F_\X\colon \X\to \X$ induced by the $p$th power map $\mathscr{O}_\X\to\mathscr{O}_\X$, given on local sections by $s\mapsto s^p$. We consider the endofunctors
\[
    \F_{\X*},\F_\X^*\colon\D^{\mathrm{b}}(\X)\to\D^{\mathrm{b}}(\X).
\]
Here, as in the rest of this note, we denote derived functors by the same symbols as their underived counterparts. As $\X$ is smooth, the morphism $\F_\X$ is finite and flat \cite[3.2]{MR1409820}, so both pushforward and pullback along $\F_\X$ are exact. 
We remark that the functors $\F_{\X*}$ and $\F_\X^*$ are not $k$-linear; rather, for objects $\E,\G\in\D^{\mathrm{b}}(\X)$, the map
\[
    \F_{\X*}:\Hom_{\D^{\mathrm{b}}(\X)}(\E,\G)\to\Hom_{\D^{\mathrm{b}}(\X)}(\F_{\X*}\E,\F_{\X*}\G)
\]
satisfies $\F_{\X*}(\lambda\varphi)=\lambda^{1/p}\F_{\X*}(\varphi)$ for $\lambda\in k$, while the map
\[
    \F_\X^*:\Hom_{\D^{\mathrm{b}}(\X)}(\E,\G)\to\Hom_{\D^{\mathrm{b}}(\X)}(\F_{\X}^*\E,\F_{\X}^*\G)
\]
satisfies $\F_{\X}^*(\lambda\varphi)=\lambda^p\F_\X^*(\varphi)$.

Let $\Y$ be another smooth projective variety over $k$ and let 
\[
    \Phi\colon\D^{\mathrm{b}}(\X)\to\D^{\mathrm{b}}(\Y)
\]
be a $k$-linear Fourier--Mukai equivalence. We say that $\Phi$ \emph{commutes with }$\F_*$ (resp. $\Phi$ \emph{commutes with }$\F^*$) if the diagram
\[
  \begin{tikzcd}
    \D^{\mathrm{b}}(\X)\arrow{r}{\Phi}\arrow{d}[swap]{\F_{\X*}}&\D^{\mathrm{b}}(\Y)\arrow{d}{\F_{\Y*}}\\
    \D^{\mathrm{b}}(\X)\arrow{r}{\Phi}&\D^{\mathrm{b}}(\Y)
  \end{tikzcd}
  \hspace{1cm}
  \mbox{resp.}
  \hspace{1cm}
    \begin{tikzcd}
    \D^{\mathrm{b}}(\X)\arrow{r}{\Phi}\arrow{d}[swap]{\F_{\X}^*}&\D^{\mathrm{b}}(\Y)\arrow{d}{\F_{\Y}^*}\\
    \D^{\mathrm{b}}(\X)\arrow{r}{\Phi}&\D^{\mathrm{b}}(\Y)
  \end{tikzcd}
\]
commutes up to a natural isomorphism. Here are some examples of equivalences which commute with both $\F_*$ and $\F^*$:
\begin{enumerate}
    \item the shift functor $[n]\colon\D^{\mathrm{b}}(\X)\to\D^{\mathrm{b}}(\X)$ for any $n\in\mathbf{Z}$,
    \item $f_*\colon\D^{\mathrm{b}}(\X)\to\D^{\mathrm{b}}(\Y)$, where $f\colon \X\isom\Y$ is an isomorphism over $k$, and
    \item $\L\otimes\rule{.25cm}{0.4pt}\colon\D^{\mathrm{b}}(\X)\to\D^{\mathrm{b}}(\X)$, where $\L$ is an invertible sheaf on $\X$ such that $\L^{\otimes p-1}\simeq\mathscr{O}_\X$.
\end{enumerate}
Indeed, the shift functor commutes with any map of triangulated categories (by definition), and the absolute Frobenius has the property that $\F_\Y\circ f=f\circ\F_\X$ for any map of schemes $f\colon\X\to\Y$. To verify the final example, we note that if $\L$ is an invertible sheaf on $\X$ then $\F_\X^*\L\simeq \L^{\otimes p}$. Therefore if $\L^{\otimes p-1}\simeq\mathscr{O}_\X$ we then have $\F_\X^*\L\simeq \L$, and so for any object $\E\in\D^{\mathrm{b}}(\X)$ we have isomorphisms
\[
    \F_\X^*(\L\otimes \E)\simeq\F_\X^*\L\otimes\F_\X^*\E\simeq\L\otimes\F_\X^*\E
\]
and
\[
    \F_{\X*}(\L\otimes \E)\simeq\F_{\X*}(\F_\X^*\L\otimes \E)\simeq \L\otimes\F_{\X*}\E
\]
which are functorial in $\E$.

In this note we will show that these are in fact the only examples.

\begin{theorem}\label{thm:main theorem}
  If $\Phi\colon\D^{\mathrm{b}}(\X)\to\D^{\mathrm{b}}(\Y)$ is a Fourier--Mukai equivalence which commutes with $\F_{*}$ or with $\F^*$, then $\Phi$ is a composition of functors of the above form. More precisely, there exists an isomorphism $f\colon \X\isom \Y$ over $k$, an invertible sheaf $\L$ on $\X$ such that $\L^{\otimes p-1}\simeq\mathscr{O}_\X$, and an integer $n$ such that
  \[
    \Phi(\rule{.25cm}{0.4pt})\simeq f_*(\L\otimes\rule{.25cm}{0.4pt})[n].
  \]
\end{theorem}

The key step is Proposition \ref{prop:picking out points}, which characterizes the shifts of structure sheaves of closed points of $\X$ among all objects of $\D^{\mathrm{b}}(\X)$ in terms of the $k$-linear triangulated category structure on $\D^{\mathrm{b}}(\X)$ together with the Frobenius pushforward endofunctor $\F_{\X*}\colon\D^{\mathrm{b}}(\X)\to\D^{\mathrm{b}}(\X)$.

\subsection{Acknowledgements}

The question of which Fourier--Mukai equivalences commute with Frobenius was asked of the author by Karl Schwede. The author received support from NSF grant \#1840190.

\section{Equivalences preserving supports}

Let $\X$ be a smooth variety over an algebraically closed field $k$ (of arbitrary characteristic). The \emph{support} of a coherent sheaf $\E$ on $\X$ is the closed subscheme of $\X$ cut out by the ideal sheaf $\I\subset\mathscr{O}_\X$ which is the kernel of the action map
\[
    \mathscr{O}_\X\to\sEnd_{\mathscr{O}_\X}(\E).
\]
Equivalently, the support of $\E$ is the minimal closed subscheme $\Z\subset \X$ such that $\E$ is the pushforward of a coherent sheaf on $\Z$. The \emph{support} of a complex $\E\in\D^{\mathrm{b}}(\X)$ is the minimal closed subscheme of $\X$ which contains the support of every cohomology sheaf of $\E$. If $\E$ is a coherent sheaf or an object of $\D^{\mathrm{b}}(\X)$, we define the \emph{reduced support} of $\E$ to be the reduced subscheme of $\X$ underlying the support.

We make the following definition.
\begin{definition}\label{def:point like}
    An object $\E\in\D^{\mathrm{b}}(\X)$ is \textit{point-like} if
    \begin{enumerate}
        \item $\Hom_{\D^{\mathrm{b}}(\X)}(\E,\E[i])=0$ for all $i<0$ and
        \item $\Hom_{\D^{\mathrm{b}}(\X)}(\E,\E)\simeq k$.
    \end{enumerate}
\end{definition}

If $x\in \X$ is a closed point then any shift $k(x)[n]$ is a point-like object of $\D^{\mathrm{b}}(\X)$. In general, there may be point-like objects with positive dimensional support. The following result shows however that every point-like object with zero dimensional support is of this form.

\begin{lemma}\label{lem:point like object and support}
    If $\E\in\D^{\mathrm{b}}(\X)$ is a point-like object with zero dimensional support, then $\E\simeq k(x)[n]$ for some closed point $x\in \X$ and integer $n$. 
\end{lemma}
\begin{proof}
     See \cite[Lemma 4.5]{Huy06}.
\end{proof}

Suppose now that $\X$ is smooth and projective. Let $\Y$ be another smooth projective variety over $k$ and let $\Phi\colon\D^{\mathrm{b}}(\X)\to\D^{\mathrm{b}}(\Y)$ be a Fourier--Mukai equivalence with kernel $\P\in\D^{\mathrm{b}}(\X\times \Y)$.
\begin{proposition}\label{prop:filtered equivalence}
  Suppose that, for every closed point $x\in \X$, the object $\Phi(k(x))\in \D^{\mathrm{b}}(\Y)$ has zero dimensional support. Then the support of $\P$ is the graph of an isomorphism $f\colon \X\isom \Y$, and there exists an integer $n$ and an invertible sheaf $\L$ on $\X$ such that
  \[
    \Phi(\rule{.25cm}{0.4pt})\simeq f_*(\L\otimes\rule{.25cm}{0.4pt})[n].
  \]
\end{proposition}
\begin{proof}
    Consider a closed point $x\in \X$. Then $k(x)$ is point-like, so $\Phi(k(x))$ is as well. By assumption, $\Phi(k(x))$ also has zero dimensional support, so by Lemma \ref{lem:point like object and support} we have $\Phi(k(x))\simeq k(y)[n]$ for some closed point $y\in \Y$ and integer $n$. By semicontinuity the integer $n$ must be independent of the point $x$ \cite[Corollary 6.14]{Huy06}. We conclude by applying \cite[Corollary 5.23]{Huy06} to the functor $[-n]\circ\Phi$.
\end{proof}

\section{Characterizing points using Frobenius}

Let $\X$ be a smooth variety over an algebraically closed field $k$ of characteristic $p>0$. Let $\Z\subset \X$ be a reduced and irreducible closed subscheme with generic point $\eta_\Z$. We say that a coherent sheaf $\E$ on $\X$ is \emph{properly supported on }$\Z$ if either $\Z$ is an irreducible component of the reduced support of $\E$ or $\E$ vanishes generically on $\Z$. Equivalently, $\E$ is properly supported on $\Z$ if the restriction $\E\otimes_{\mathscr{O}_\X}\mathscr{O}_{\X,\eta_\Z}$ of $\E$ to the local scheme $\Spec \mathscr{O}_{\X,\eta_\Z}$ has finite length. If $\E$ is a coherent sheaf on $\X$ which is properly supported on $\Z$, we define
\[
    \rk_\Z(\E):=\len_{\mathscr{O}_{\X,\eta_\Z}}(\E\otimes_{\mathscr{O}_\X}\mathscr{O}_{\X,\eta_\Z}).
\]
The key property of this quantity that we will use is that it is additive in short exact sequences of coherent sheaves properly supported on $\Z$.

\begin{lemma}\label{lem:supports of pushforwards}
    If $\Z\subset \X$ is a reduced and irreducible closed subvariety of dimension $d$ and $\E$ is a coherent sheaf on $\X$ which is properly supported on $\Z$, then we have
    \[
        \rk_\Z(\F_{\X*}\E)=p^d\rk_\Z(\E).
    \]

\end{lemma}
\begin{proof}
    Let $\E$ be a coherent sheaf properly supported on $\Z$. Let $\Z'$ be the support of $\E$. We first consider the special case when $\Z=\X$. By passing to an open subscheme, we may assume that $\E$ is free, and then further reduce to the case when $\E=\mathscr{O}_\X$. The result then follows from the fact that the absolute Frobenius of a smooth $k$-variety of dimension $d$ is finite and locally free of degree $p^d$ \cite[Proposition 3.2]{MR1409820}.

    Next we consider the special case when $\Z'\subset \Z$. By passing to an open subscheme we may assume that $\Z$ is smooth. Let $i\colon \Z\hookrightarrow \X$ be the inclusion. We have a commutative diagram
    \[
        \begin{tikzcd}
            \Z\arrow[hook]{d}[swap]{i}\arrow{r}{\F_\Z}&Z\arrow[hook]{d}{i}\\
            \X\arrow{r}{\F_\X}&\X.
        \end{tikzcd}
    \]
    Write $\E_\Z=i^*\E$ for the (non-derived) pullback of $\E$ to $\Z$. As $\Z'\subset \Z$, we have that $\E=i_*\E_\Z$, and therefore
    \[
        \F_{\X*}\E=\F_{\X*}(i_*\E_\Z)=i_*(\F_{\Z*}\E_\Z).
    \]
    From the previous case applied to $\E_\Z$ we obtain the equality $\rk_\Z(\F_{\Z*}\E_\Z)=p^d\rk_\Z(\E_\Z)$, which implies that 
    \[
        \rk_\Z(i_*(\F_{\Z*}\E_\Z))=p^d\rk_\Z(i_*\E_\Z).
    \]
    Combining these we obtain that
    \[
        \rk_\Z(\F_{\X*}\E)=\rk_\Z(i_*(\F_{\Z*}\E_\Z))=p^d\rk_\Z(i_*\E_\Z)=p^d\rk_\Z(\E)
    \]
    as claimed.
    
    Finally we consider the general case. Let $\I$ be the ideal sheaf of $\Z$. If $\E$ is properly supported on $\Z$, then for all sufficiently large integers $m$ we have that $\rk_\Z(\I^m\E)=0$ (equivalently, $\I^m\E$ vanishes generically on $\Z$). We induct on $m$ the claimed statement for those coherent sheaves $\E$ which are properly supported on $\Z$ and satisfy $\rk_\Z(\I^m\E)=0$. For the base case, suppose that $\E$ is properly supported on $\Z$ and satisfies $\rk_\Z(\I\,\E)=0$. Then after passing to an open subscheme we may ensure that $\I\,\E=0$, in which case the support of $\E$ is contained in $\Z$, and so we are in the situation already dealt with above. For the induction step, suppose that $\E$ is properly supported on $\Z$ and that $\rk_\Z(\I^{m+1}\E)=0$. Consider the short exact sequence
    \[
        0\to \I\,\E\to \E\to \E/\I\,\E\to 0.
    \]
    The left and right terms are again properly supported on $\Z$, and moreover each of their products with $\I^m$ has $\rk_\Z=0$. Applying $\F_{\X*}$ and using the additivity of $\rk_\Z$ and the induction hypothesis, we get
    \begin{align*}
        \rk_\Z\left(\F_{\X*}\E\right)&=\rk_\Z\left(\F_{\X*}\I\,\E\right)+\rk_\Z\left(\F_{\X*}(\E/\I\,\E)\right)\\
        &=p^d\rk_\Z(\I\,\E)+p^d\rk_\Z\left(\E/\I\,\E)\right)\\
        &=p^d\rk_\Z(\E).
    \end{align*}
\end{proof}

\begin{lemma}\label{lem:equal to your own pushforward}
    If $\E\in\D^{\mathrm{b}}(\X)$ is a nonzero object such that $\E\simeq\F_{\X*}\E$, then $\E$ has zero dimensional support.
\end{lemma}
\begin{proof}
    If $\E\simeq\F_{\X*}\E$, then the cohomology sheaves of $\E$ satisfy $\mathscr{H}^i(\E)\simeq\mathscr{H}^i(\F_{\X*}\E)=\F_{\X*}\mathscr{H}^i(\E)$ for each integer $i$. We therefore reduce to the case when $\E$ is a nonzero coherent sheaf. Let $\Z$ be the reduction of an irreducible component of the support of $\E$. Then $\E$ is properly supported on $\Z$ and $\rk_\Z(\E)\neq 0$. By Lemma \ref{lem:supports of pushforwards} we have
    \[
        \rk_\Z(\E)=\rk_\Z(\F_{\X*}\E)=p^d\rk_\Z(\E)
    \]
    where $d$ is the dimension of $\Z$. It follows that $d=0$.
\end{proof}

Combining the above results we obtain the following characterization of the shifts of structure sheaves of points in $\D^{\mathrm{b}}(\X)$.

\begin{proposition}\label{prop:picking out points}
  Let $\E\in\D^{\mathrm{b}}(\X)$ be an object. The following are equivalent.
  \begin{enumerate}
      \item $\E\simeq k(x)[n]$ for some closed point $x\in \X$ and integer $n$.
      \item $\E$ is point-like and $\E\simeq\F_{\X*}\E$.
  \end{enumerate}
\end{proposition}
\begin{proof}
    (1) implies (2) is immediate. Conversely, if $\E\simeq\F_{\X*}\E$, then $\E$ has zero dimensional support by Lemma \ref{lem:equal to your own pushforward}. Thus, if $\E$ is also point-like, then by Lemma \ref{lem:point like object and support} we have $\E\simeq k(x)[n]$ for some closed point $x\in \X$ and integer $n$.
\end{proof}

\section{Proof of Theorem \ref{thm:main theorem}}

We recall the notation: $\X$ and $\Y$ are smooth projective varieties over an algebraically closed field $k$ of characteristic $p>0$ and $\Phi\colon\D^{\mathrm{b}}(\X)\to\D^{\mathrm{b}}(\Y)$ is a Fourier--Mukai equivalence.

\begin{lemma}\label{lem:one implies the other}
  The equivalence $\Phi$ commutes with $\F^*$ if and only if it commutes with $\F_*$.
\end{lemma}
\begin{proof}
    We have the adjunction
    \[
        \Hom_{\D^{\mathrm{b}}(\X)}(\F_{\X}^*\E,\G)\simeq\Hom_{\D^{\mathrm{b}}(\X)}(\E,\F_{\X*}\G)
    \]
    for objects $\E,\G\in\D^{\mathrm{b}}(\X)$. Because $\Phi$ is an equivalence, this gives rise to isomorphisms
    \[
        \Hom_{\D^{\mathrm{b}}(\Y)}(\Phi(\F_{\X}^*\E),\Phi(\G))\simeq\Hom_{\D^{\mathrm{b}}(\Y)}(\Phi(\E),\Phi(\F_{\X*}\G))
    \]
    which are functorial in $\E$ and $\G$. Suppose that $\F_{\Y}^*\circ\Phi\simeq\Phi\circ \F_{\X}^*$. Then we obtain functorial isomorphisms
    \begin{align*}
        \Hom_{\D^{\mathrm{b}}(\Y)}(\Phi(\E),\F_{\Y*}\Phi(\G))&\simeq\Hom_{\D^{\mathrm{b}}(\Y)}(\F_{\Y}^*\Phi(\E),\Phi(\G))\\
        &\simeq\Hom_{\D^{\mathrm{b}}(\Y)}(\Phi(\F_\X^*\E),\Phi(\G))\\
        &\simeq\Hom_{\D^{\mathrm{b}}(\Y)}(\Phi(\E),\Phi(\F_{\X*}\G)).
    \end{align*}
    As $\Phi$ is an equivalence, Yoneda's Lemma implies that $\F_{\Y*}\circ\Phi\simeq\Phi\circ\F_{\X*}$. The reverse implication is similar.
\end{proof}

\begin{proof}[Proof of Theorem \ref{thm:main theorem}]
By Lemma \ref{lem:one implies the other} it suffices to consider the case when $\Phi$ commutes with $\F_*$. For a closed point $x\in \X$, we have that $k(x)$ and hence $\Phi(k(x))$ are point-like, and furthermore
\[
    \Phi(k(x))\simeq\Phi(\F_{\X*}k(x))\simeq\F_{\Y*}\Phi(k(x)).
\]
Proposition \ref{prop:picking out points} implies that $\Phi(k(x))\simeq k(y)[n]$ for some closed point $y\in \Y$ and some integer $n$, and in particular $\Phi(k(x))$ has zero dimensional support. By Proposition \ref{prop:filtered equivalence} there exists an isomorphism $f\colon \X\isom \Y$, an integer $n$, and an invertible sheaf $\L$ on $\X$ such that
\[
    \Phi(\rule{.25cm}{0.4pt})\simeq f_*(\L\otimes\rule{.25cm}{0.4pt})[n].
\]
It remains to show that $\L^{\otimes p-1}\simeq\mathscr{O}_\X$. To see this, we note that shifts and pushforwards along isomorphisms always commute with both $\F^*$ and $\F_*$. It follows that tensoring with $\L$ commutes with $\F_*$, hence also with $\F^*$, and therefore we have
\[
    \L\otimes\F^*_\X\E\simeq\F^*_\X(\L\otimes \E)
\]
for every $\E\in\D^{\mathrm{b}}(\X)$. Taking $\E=\mathscr{O}_\X$ we conclude that $\L\simeq\L^{\otimes p}$, and thus $\mathscr{O}_{\X}\simeq\L^{\otimes p-1}$.
\end{proof}

\bibliographystyle{plain}
\bibliography{biblio}

\begin{thebibliography}{1}

\bibitem{Huy06}
Daniel Huybrechts.
\newblock {\em Fourier-{M}ukai transforms in algebraic geometry}.
\newblock Oxford Mathematical Monographs. The Clarendon Press, Oxford
  University Press, Oxford, 2006.

\bibitem{MR1409820}
Luc Illusie.
\newblock Frobenius et d\'{e}g\'{e}n\'{e}rescence de {H}odge.
\newblock In {\em Introduction \`a la th\'{e}orie de {H}odge}, volume~3 of {\em
  Panor. Synth\`eses}, pages 113--168. Soc. Math. France, Paris, 1996.

\end{thebibliography}

\end{document}